\documentclass[12pt]{article}
\usepackage{amssymb}
\usepackage{amsmath}
\usepackage{graphicx}

\begin{document}
\newtheorem{theorem}{Theorem} [section]
\newtheorem{tha}{Theorem}
\newtheorem{conjecture}[theorem]{Conjecture}
\renewcommand{\thetha}{\Alph{tha}}

\newtheorem{corollary}[theorem]{Corollary}
\newtheorem{lemma}[theorem]{Lemma}
\newtheorem{proposition}[theorem]{Proposition}
\newtheorem{construction}[theorem]{Construction}
\newtheorem{claim}[theorem]{Claim}
\newtheorem{definition}[theorem]{Definition}
\newtheorem{observation}{Observation}
\newtheorem{question}{Question}
\newtheorem{remark}[theorem]{Remark}
\newtheorem{algorithm}[theorem]{Algorithm}
\newtheorem{example}[theorem]{Example}
\newtheorem{problem}[theorem]{Problem}

\def\endproofbox{\hskip 1.3em\hfill\rule{6pt}{6pt}}
\newenvironment{proof}%
{%
\noindent{\it Proof.}
}%
{%
 \quad\hfill\endproofbox\vspace*{2ex}
}
\def\qed{\hskip 1.3em\hfill\rule{6pt}{6pt}}

\parindent=15pt

\def\ce#1{\lceil #1 \rceil}
\def\fl#1{\lfloor #1 \rfloor}
\def\dim{{\rm dim}}
\def \e{\epsilon}
\def\vec#1{{\bf #1}}
\def\va{\vec{a}}
\def\vb{\vec{b}}
\def\vc{\vec{c}}
\def\vx{\vec{x}}
\def\vy{\vec{y}}
\def\cA{{\mathcal A}}
\def\cB{{\mathcal B}}
\def\cC{{\mathcal C}}
\def\cD{{\mathcal D}}
\def\cM{{\mathcal M}}
\def\cH{{\cal H}}
\def\cF{{\mathcal F}}
\def\bE{\mathbb{E}}
\def\bP{\mathbb{P}}
\def\cA{{\mathcal A}}
\def\cB{{\mathcal B}}
\def\cJ{{\mathcal J}}
\def\cM{{\mathcal M}}
\def\cS{{\mathcal S}}
\def\cT{{\mathcal T}}
\def\wt#1{\widetilde{#1}}

\title{On the Bandwidth of the Kneser graph}
\author{ Tao Jiang \thanks{ Dept. of Mathematics, Miami University,
Oxford, OH 45056, USA, jiangt@miamioh.edu} \,
Zevi Miller \thanks{Dept. of
Mathematics, Miami University, Oxford, OH 45056, USA,
millerz@miamioh.edu}\, Derrek Yager \thanks{Dept. of Mathematics, University of Illinois, 
Champaign-Urbana, Illinois, yager2@illinois.edu} }

\date{November 30, 2015}
\maketitle

\begin{abstract}

Let $G = (V,E)$ be a graph on $n$ vertices and $f: V\rightarrow [1,n]$ a one to one map of $V$ onto the integers $1$ 
through $n$.  Let $dilation(f) =$ max$\{ |f(v) - f(w)|: vw\in E \}$.  Define the {\it bandwidth} $B(G)$ of $G$ to be the minimum possible 
value of $dilation(f)$ over all such one to one maps $f$.  Next define the {\it Kneser Graph} $K(n,r)$ to be the graph with vertex set $\binom{[n]}{r}$, 
the collection of $r$-subsets of an $n$ element set, and edge set $E = \{ vw: v,w\in \binom{[n]}{r}, v\cap w = \emptyset  \}$.  For fixed $r\geq 4$ and 
$n\rightarrow \infty$ we show that 
$$B(K(n,r)) = \binom{n}{r} - \frac{1}{2}\binom{n-1}{r-1} - 2\frac{n^{r-2}}{(r-2)!} + (r + 2)\frac{n^{r-3}}{(r-3)!} + O(n^{r-4}).$$    

\end{abstract}

\section{Introduction}\label{intro}

We begin with some notation.  Let $[n] = \{ 1,2,3,\cdots, n\}$, which we view as our canonical set of 
size $n$.  Let $\binom{[n]}{r}$ be the collection of $r$-subsets of $[n]$, 
and for $k\leq r$ let $\binom{[k]}{r}$ denote the collection of all $r$-subsets of an 
arbitrary set $S\subset [n]$ of size $k$ (so $S$ is not necessarily $ \{ 1,2,3,\cdots, k\}$).    
For integers $a < b$ we let $[a,b]$ denote the set of integers $x$ satisfying $a\le x\le b$.  

Now let $\mathcal{A}$ and $\mathcal{B}$ be two families of subsets of $[n]$.  We say $\mathcal{A}$ is  
{\it intersecting} if $A_{1}\cap A_{2} \ne \emptyset$ for all pairs $A_{1}, A_{2}\in \mathcal{A}$.  Further $\mathcal{A}$ is {\it nontrivial} 
if $\cap_{A\in \mathcal{A}}A = \emptyset$, and is {\it trivial} otherwise.  
The pair of families $\mathcal{A}$, $\mathcal{B}$ 
 is {\it cross intersecting} if $A\cap B\ne \emptyset$ for all pairs of sets $A,B$, where $A\in \mathcal{A}$ and $B\in \mathcal{B}$.  A 
{\it matching} of $\mathcal{A}$ is a collection of sets in $\mathcal{A}$ that are pairwise disjoint. 
For $S\subset [n]$ 
we let $V(S) = \{x\in [n]:  x\in S \}$, and we let $V(\mathcal{A}) = \cup_{S\in \mathcal{A}}V(S)$ (the vertex 
set of $\mathcal{A}$).  We sometimes refer to the members of $\mathcal{A}$ as {\it edges} of $\mathcal{A}$.     
We will use standard 
graph theoretic or combinatorial notation, as may be found for example in \cite{W}.  Additional notation will be defined where it is initially used in the text.

Now define the the 
{\it Kneser Graph} $K(n,r)$ to be the graph with vertex set $V = \binom{[n]}{r}$, 
and edge set $E = \{ vw: v,w\in \binom{[n]}{r}, v\cap w = \emptyset  \}$.  We can suppose that $n\geq 2r$ since otherwise $K(n,r)$ 
has no edges.  Clearly $K(n,r)$ has $\binom{n}{r}$ vertices, is regular of degree $\binom{n-r}{r}$, and it can be shown that it is both vertex and edge transitive.  
The Kneser Graph arises in 
several examples; $K(n,1)$ is just the complete graph $K_{n}$ on $n$ vertices, $K(n,2)$ is the complement of the line graph of $K_{n}$, 
$K(2n-1,n-1)$ is also known as the odd graph $O_{n}$, and $K(5,2)$ is isomorphic to the Petersen graph.  The diameter of $K(n,r)$ was shown to be 
$\lceil \frac{r-1}{n-2r} \rceil + 1$ in \cite{Val}, and $K(n,r)$ was shown to be Hamiltonian for $n\geq \frac{1}{2}( 3r + 1 + \sqrt{5r^{2} -2r +1} )$ in \cite{Chen}.   

A longstanding problem on $K(n,r)$ was Kneser's conjecture; that the chromatic number satisfies 
$\chi(K(n,r)) = n -2r + 2$ if $n\geq 2r$ and of course $\chi(K(n,r)) = 1$ otherwise.  The upper bound is achieved by a simple 
coloring; color an $r$-set by its largest element if this element is at least $2r$, and otherwise color it by $1$.  The difficulty was in proving the 
corresponding lower bound, and this result was first proved by Lovasz \cite{L} using methods of algebraic topology.  More elementary, but still topological, proofs were given 
by B\'{a}r\'{a}ny \cite{Ba} soon after, and by Dol'nikov \cite{Do} and Greene \cite{Gr} later.  A mostly combinatorial proof (still with topological elements) was given by Matou\v{s}ek \cite{Ma}.

Recently some results on a labeling problem relating to $K(n,r)$ appeared in the literature \cite{JPP}.  Let $G = (V,E)$ be a graph on $n$ vertices 
and $f: V\rightarrow C_{n}$ a one to one map of the vertices of $G$ to the cycle $C_{n}$ on $n$ vertices.  Let $|f| = $ min$\{ dist_{C_{n}}(f(u), f(v)): uv\in E \}$, where 
$dist_{C_{n}}$ denotes the distance function on $C_{n}$; that is, $dist_{C_{n}}(x,y)$ is the mod $n$ distance between $x$ and $y$ when we view 
the vertices of $C_{n}$ as the integers mod $n$.  Now let $s(G) = $ max$\{ |f| \}$, where the maximum is taken over all such one to one maps $f$.  It is shown 
in \cite{JPP} that $s(K(n,2)) = 3$ when $n\geq 6$, that $s(K(n,3)) = 2n-7$ or $2n-8$ for $n$ sufficiently large, and that for fixed $r\geq 4$ and $n$ sufficiently large we have 
$\frac{2n^{r-2}}{(r-2)!} - \frac{(\frac{7}{2}r - 2)n^{r-3}}{(r-3)!} - O(n^{r-4})\le s(K(n,r))\le \frac{2n^{r-2}}{(r-2)!} - \frac{(\frac{7}{2}r - 3.2)n^{r-3}}{(r-3)!} + o(n^{r-3}).$

This paper considers the following related well known labeling problem.  Let $G = (V,E)$ be a graph on $n$ vertices.  Now consider $f: V\rightarrow [1,n]$ a 
one to one map, and let $dilation(f) =$ max$\{ |f(v) - f(w)|: vw\in E \}$.  Define the {\it bandwidth} $B(G)$ of $G$ to be the minimum possible 
value of $dilation(f)$ over all such one to one maps $f$.  There is an extensive literature on the bandwidth of graphs and related layout problems (see \cite{DPS} for a survey), 
originally motivated by the attempt 
to find fast algorithms for matrix operations, and by problems in VLSI design.  Our main result is the following.

\begin{theorem}  \label{main theorem}
Let $r\geq 4$ be fixed integer.  As $n\rightarrow \infty$ we have
$$B(K(n,r)) = \binom{n}{r} - \frac{1}{2}\binom{n-1}{r-1} - 2\frac{n^{r-2}}{(r-2)!} + (r + 2)\frac{n^{r-3}}{(r-3)!} + O(n^{r-4}).$$ 
\end{theorem}

We observe that there is the trivial upper bound $B(K(n,r))\le \binom{n}{r} - \frac{1}{2}\binom{n-1}{r-1}$, as follows.  Let $\beta(G)$ be 
the maximum possible size of an independent set of vertices in any graph $G$ on $N$ vertices.  
Then $B(G)\le N - \lfloor  \frac{1}{2}\beta(G) \rfloor$, achieved by a one to one map $f: V(G)\rightarrow [1,N]$ which sends any half of the vertices of a 
maximum independent set $S$ to $[1, \lfloor  \frac{1}{2}\beta(G) \rfloor]$, the other half of $S$ to 
$[N- \lceil \frac{1}{2}\beta(G)  \rceil+1, N]$, and the remainder of $V(G)$ arbitrarily to the rest of the interval $[1,N]$.  Now an independent set in $K(n,r)$ 
is just an intersecting family in $\binom{[n]}{r}$, and by the Erdos, Ko, Rado theorem \cite{EKR} the maximum size of such a family is $\binom{n-1}{r-1}$.  It follows 
that  $B(K(n,r))\le \binom{n}{r} - \frac{1}{2}\binom{n-1}{r-1}$.  

Our contribution here is to precisely determine $B(K(n,r))$ for fixed $r$ and $n$ growing, up to an $O(n^{r-4})$ error term.  With this in mind, we will occasionally state inequalities 
involving $n$ and $r$ which are true when $n$ is large enough relative to the fixed $r$.  In these cases we will often not state this requirement on $n$ and $r$ explicitly.

\section{The lower bound}

Our goal in this section is to prove the lower bound $B(K(n,r)) \geq \binom{n}{r} - A$, where $A = \frac{1}{2}\binom{n-1}{r-1} + 2\frac{n^{r-2}}{(r-2)!} - (r + 2)\frac{n^{r-3}}{(r-3)!} + O(n^{r-4})$; 
that is, to show that $dilation(f)\geq \binom{n}{r} - A$ for any one to one map $f: V(K(n,r))\rightarrow [1, \binom{n}{r}]$.    

As notation, for two families of $r$-sets $\cA,\cB\subset \binom{[n]}{r}$, let us write $\cA\sim \cB$ to mean that there is 
some $S\in \cA$ and  $T\in \cB$ such that $S\cap T = \emptyset$.  So $\cA\sim \cB$ says that $\cA$ and $\cB$ are not cross intersecting, or equivalently that there are vertices $S,T$ of $K(n,r)$, where 
$S\in \cA$ and  $T\in \cB$, such that $ST\in E(K(n,r))$.  Roughly speaking we will be showing that for any one to one map $f: V(K(n,r))\rightarrow [1, \binom{n}{r}]$ there is 
an initial (resp. final) subinterval $I$ (resp. $F$) of $[1, \binom{n}{r}]$, with $|I| + |F|$ 
reasonably small, such that $f^{-1}(I)\sim f^{-1}(J)$.  This forces a ``long" edge $ST$; that is one satisfying $|f(S) - f(T)|\geq \binom{n}{r} - (|I| + |F|)$, and leads to our lower bound 
on $B(K(n,r))$.

We discuss briefly the relation between our lower bound proof and existing results in the literature on cross intersecting families. 
Now $dilation(f)\geq \binom{n}{r} - A$ is equivalent to at least one of statements $f^{-1}(j)\sim f^{-1}([\binom{n}{r} - A + j, \binom{n}{r}])$, $1\le j\le A$, being true.   This is in turn equivalent to  
at least one of the statements $f^{-1}([1,j])\sim f^{-1}([\binom{n}{r} - A + j, \binom{n}{r}])$, $1\le j\le A$, being true.  In these and also other cases
we will be interested in proving $\cA\sim \cB$ for certain pairs $\cA,\cB$ of families of subsets of $[n]$.  There are many results in the 
literature which say that if $\cA,\cB$ are cross intersecting families (possibly satisfying additional conditions), then $|\cA| + |\cB|\le U(n)$ or $|\cA||\cB|\le T(n)$ for suitable functions 
U and T.  We mention the papers \cite{HM}, \cite{FT}, and \cite{FranklTok} as examples of the large literature containing such bounds.  
If we could show that the families $\cA,\cB$ for which we try to prove $\cA\sim \cB$ violate these bounds; that is, $|\cA| + |\cB| > U(n)$ for example, then we could conclude that  $\cA,\cB$ could not be cross intersecting and hence  
that $\cA\sim \cB$ as desired.  In the examples just cited, the bounds were 
too generous for our purposes, and to the best of our knowledge the same holds for the other published results of this type.
Indeed in our setting, we will be dealing with
cross-intersecting families $\cA,\cB$ where one of them contains a large matching,
which substantially restricts $|\cA|+|\cB|$.

We now proceed to our lower bound result. We will use a few results from the literature. The following extension
of the Hilton-Milner theorem on intersecting families to cross-intersecting families was established by F\"uredi \cite{F-cross-int}.

\begin{theorem} {\rm \cite{F-cross-int}} \label{cross-intersecting}
Let $n,a,b$ be positive integers where $n\geq a+b$. Let $\cA\subseteq \binom{[n]}{a}$
and $\cB\subseteq \binom{[n]}{b}$, and suppose that $\cA$ and $\cB$ are cross-intersecting.  
 If $|\cA|\geq \binom{n-1}{a-1}-\binom{n-b-1}{a-1}+1$ and $|\cB|>\binom{n-1}{b-1}-\binom{n-a-1}{b-1}+1$, then there exists an element $x\in [n]$ that lies in all members of $\cA$ and $\cB$.  
 That is; $\cA \cup \cB$ is a trivial family.
\end{theorem}

Let $t$ be a positive integer. A family $\cF$ of sets is said to be {\it $t$-intersecting}
if $|F\cap F'|\geq t$ for all $F, F'\in\cF$. Erd\H{os}, Ko, and Rado \cite{EKR} proved that for
fixed positive intergers $r,t$, where $r\geq t+1$, there exists
$n_0(r,t)$ such that if $n\geq n_0(r,t)$ then the maximum size of an $t$-intersecting
family of $r$-subsets of $[n]$ is $\binom{n-t}{r-t}$. For $t\geq 15$, Frankl \cite{Frankl-t-intersecting} obtained the smallest possible $n_0(r,t)$ for which the statement holds.
Wilson \cite{Wilson} obtained the smallest possible $n_0(r,t)$ for all $t$.
\begin{theorem} {\rm (\cite{Frankl-t-intersecting} for $t\geq 15$, \cite{Wilson} for all $t$)} \label{t-intersecting}
For all $n\geq (r-t+1)(t+1)$, if $\cF\subseteq \binom{[n]}{r}$ satisfies that
$|F\cap F'|\geq t$ for all $F,F'\in \cF$ (i.e. $\cF$ is $t$-intersecting) then $|\cF|\leq \binom{n-t}{r-t}$.
\end{theorem}

Erd\H{o}s \cite{Erdos-matching} showed that there exist $n_1(r,p)$ such that
for all $n\geq n_1(r,p)$ the maximum size of a family of $r$-subsets of $[n]$
not containing a matching of $p+1$ edges is $\binom{n}{r}-\binom{n-p}{r}$.
There has subsequently been a lot of work on
determining the smallest $n_1(r,p)$ for which the statement holds (see \cite{BDE, FLM, FRR, HLS} for instance). The best result among these is due to Frankl \cite{Frankl-matching}.

\begin{theorem}  {\rm \cite{Frankl-matching}} \label{F-matching}
Let $\cF\subseteq \binom{[n]}{r}$ such that $\cF$ contains no matching of size $p+1$,
where $n\geq (2p+1)r-p$. Then $|\cF|\leq \binom{n}{r}-\binom{n-p}{r}$.
\end{theorem}

Note that $\binom{n}{r}-\binom{n-p}{r}<p\binom{n-1}{r-1}$. For our purpose, we will just use the following weakening of Theorem \ref{F-matching} that applies to all $n$.

\begin{lemma} \label{matching} {\rm \cite{Frankl-matching}}
Suppose $\cF\subseteq \binom{[n]}{r}$ satisfies $|\cF|> p\binom{n-1}{r-1}$. Then $\cF$ contains a matching of size $p+1$.
\end{lemma}

In fact, Frankl showed that if $\cF\subseteq \binom{n}{r}$ contains no $(s+1)$-matching
then $|\cF|\leq s|\delta(\cF)|$, where $\delta(\cF)$ denotes the number of distinct $(r-1)$-sets that are contained in edges of $\cF$.

The following lemma is straightforward to verify.

\begin{lemma} \label{binomial-estimate}
Let $n,r,c$ be integers, where $n\geq r\geq 2$ and $c\leq r$. We have
$$\frac{n^r}{r!}-\frac{c+\frac{r-1}{2}}{(r-1)!} n^{r-1}
\leq \binom{n-c}{r}\leq \frac{n^r}{r!}-\frac{c+\frac{r-1}{2}}{(r-1)!} n^{r-1}
+4r^4 n^{r-2}.$$
\end{lemma}

We can now prove our lower bound result.  Through the remainder of this section, for each 
vertex $x\in V(K(n,r))$, let $D(x)$ denote the $r$-subset of $[n]$ to which $x$ corresponds.

\begin{theorem}  \label{lower-bound}
Let $r\geq 4$ be a fixed positive integer.  
Let $f$ be a  bijection from $V(K(n,r))$ to $\{1,\ldots, \binom{n}{r}\}$. Then for $n$ sufficiently large relative to $r$ we have 
$$dilation(f)\geq \binom{n}{r}-\frac{1}{2}\binom{n-1}{r-1}-2\frac{n^{r-2}}{(r-2)!}+\frac{r+2}{(r-3)!} n^{r-3}-9r^4 n^{r-4}.$$
\end{theorem}
\begin{proof}
Let $$N=\ce{\frac{1}{2}\binom{n-1}{r-1}}-r\binom{n-2}{r-2}.$$
and
$$\cA=\left\{D(x):1\leq f(x) \le N\right \} \mbox { and }
\cB=\left\{D(x):\binom{n}{r}-N\leq f(x)\leq
\binom{n}{r}\right\}.$$
Assume $n$ to be sufficiently large relative to $r$ so that $\cA$ and $\cB$ are disjoint.
If $dilation(f)\geq \binom{n}{r}-\frac{1}{2}\binom{n-1}{r-1} + K$ for some constant $K$, then we are already done. So we assume that $dilation(f)<\binom{n}{r}-\frac{1}{2}\binom{n-1}{r-1} + K$ 
for any constant $K$.

\medskip

{\bf Claim 1.} There exists an element $x$ that lies in all members of $\cA \cup \cB$; that is $\cA \cup \cB$ is trivial. 

\medskip

{\it Proof of Claim 1.} Let $\cA'=\{D(x): f(x)\leq r\binom{n-2}{r-2}\}$ and $\cB'=\{D(x):
f(x)\geq \binom{n}{r}-r\binom{n-2}{r-2}\}$. Then $\cA'\subseteq \cA$ and $\cB'\subseteq
\cB$. If $\cA'$ and $\cB$ are not cross-intersecting then there exist $x,y$ in $K(n,r)$
with $f(x)\leq r\binom{n-2}{r-2}$ and $f(y)\geq \binom{n}{r}-N$ such that
$xy\in E(K(n,r))$, which yields $dilation(f)\geq |f(x)-f(y)|\geq \binom{n}{r}-N-r\binom{n-2}{r-2}\geq \binom{n}{r}-\frac{1}{2}{\binom{n-1}{r-1}} - 1$, contradicting
our assumption. Hence $\cA',\cB$ are cross-intersecting. Similarly, $\cA, \cB'$ 
are cross-intersecing. For sufficiently large $n$, we have
$|\cA'|=r\binom{n-2}{r-2}>\binom{n-1}{r-1}-\binom{n-r-1}{r-1}+1$
and $|\cB|\geq\frac{1}{2}\binom{n-1}{r-1}-r\binom{n-2}{r-2} - 1 >\binom{n-1}{r-1}-\binom{n-r-1}{r-1}+1$.
By Theorem \ref{cross-intersecting}, there exists an element $x$ that lies in 
all members of $\cA'$ and $\cB$. By a similar argument, there exists an element
$y$ that lies in all members of $\cA$ and $\cB'$. Suppose $x\neq y$. Then
$x,y$ both lie in all members of $\cA'$, which is impossible since there are
only $\binom{n-2}{r-2}$ many $r$-subsets of $[n]$ containing both $x$ and $y$, while
$|\cA'|=r\binom{n-2}{r-2}\geq 4\binom{n-2}{r-2}$. Hence $x=y$. So the element $x$ lies in all $r$-subsets of 
$\cA \cup \cB$.  \qed

\medskip

Without loss of generality, we may assume that element $1$ lies in all members of
$\cA \cup \cB$. Let $\cS_1=\{D\in \binom{[n]}{r}, 1\in D\}$ and $\overline{\cS_1} = \{D\in \binom{[n]}{r}: 1\notin D\}$. Then
$\cA\cup \cB\subseteq \cS_1$, $|\cS_1|=\binom{n-1}{r-1}$, and $|\overline{\cS_1} |=\binom{n-1}{r}.$

For every pair of elements $x,y\in [n]$, let $a(x,y)$ denote the number of sets in
$\cA$ that contain either $x$ or $y$ or both, and let $b(x,y)$ denote the number
of sets in $\cB$ that contain either $x$ or $y$ or both.  Also let $\cA_1(x,y) =\{D(z): 1\leq f(z)\leq  \max\{\binom{n-2}{r-2}+3r^2\binom{n-3}{r-3}, a(x,y)\} + \binom{n-3}{r-3}\}$   

\medskip

{\bf Claim 2.}
Let $\mathcal{X} \subseteq \overline{\cS_1}$ satisfy 
$|\mathcal{X}|>\binom{n-3}{r-3}+r^2\binom{n-4}{r-4}$. Then there exist elements $u,v\in [n]$ such 
that  $\cA_1(u,v)$ and $\mathcal{X}$ are not cross-intersecting.

\medskip

{\it Proof of Claim 2.} Let $\cA_0=\{D(x): 1\leq f(x)\leq \binom{n-2}{r-2}+3r^2\binom{n-3}{r-3}\}$ and $\cA'_0=\{D\setminus \{1\}: D\in \cA_0\}$, so $|\cA'_0| = |\cA_0|$.
Since  $\cA_0\subseteq \cA_{1}(u,v)$ for any $u,v\in [n]$,  
we may assume that $\cA_0$ and $\mathcal{X}$ are cross-intersecting, since otherwise the claim holds for any choice of $u,v$.
First suppose $\cA'_0$ contains a matching $M$ of size $4$. Let $C\in \mathcal{X} $. Since
$1\notin C$, $C$ must intersect each of the edges of $M$. Since $M$ is a matching,
$C$ intersects each edge of $M$ in a different vertex. But the number of such $C$
is at most $(r-1)^4 \binom{n-4}{r-4}<|\mathcal{X}|$ for large $n$, a contradiction.  Thus any maximum matching $M$ of $\cA'_0$ 
satisfies $|M|\le 3$, so that $|V(M)|\le 3(r-1) < 3r$.  
Note that $V(M)$ forms a vertex cover of $\cA'_0$; or else
we can find a larger matching in $\cA'_0$. Hence some vertex $u$ in $V(M)$ lies in
at least $|\cA'_0|/3r>3r^{2}\binom{n-3}{r-3}/3r>r\binom{n-3}{r-3}$ edges of $\cA'_0$.  Let $\cA'_{0}(u)$ be the set 
of edges in $\cA'_0$ that contain $u$, and let $\cA''_{0}(u) = \{E-u: E\in \cA'_{0}(u)\}$.  Then $\cA''_{0}(u)\subset \binom{[n-2]}{r-2}$ and and we've seen that
$|\cA''_{0}(u)| > r\binom{n-3}{r-3}$.  Applying Lemma \ref{matching} we see that $\cA''_{0}(u)$ contains a matching of size $r+1$, call it $D_{1}, D_{2}, \cdots, D_{r+1}$.  
Let $E_{i} = D_{i}\cup \{u\} \in \cA'_0$, $1\le i\le r+1$.
Note that $u$ lies in at most $\binom{n-2}{r-2}$ edges of
$\cA'_0$.  So there are at least $|\cA'_0| - \binom{n-2}{r-2} = 3r^2\binom{n-3}{r-3}$ edges 
of $\cA'_0$ not covered by $u$. These edges are covered by $V(M)\setminus \{u\}$. So some vertex $v\in V(M)\setminus \{u\}$ must lie in at least $3r^2\binom{n-3}{r-3}/(3r-1)\geq r\binom{n-3}{r-3}$ of these edges. Hence by Lemma \ref{matching} as above, there are $(r+1)$ edges $F_1,\ldots, F_{r+1}$ of $\cA'_0$ containing $v$ such that $F_1\setminus \{v\},\ldots,
F_{r+1}\setminus \{v\}$ form a matching in $\binom{[n-2]}{r-2}$. Again let $C\in \mathcal{X} $. Since $1\notin C$ by assumption, $C$ must intersect all members of $\cA'_0$; in particular, $C$ intersects 
all the edges $E_1,\ldots, E_{r+1}$. Since $|C|\leq r$, it follows that $u\in C$. Similarly,
we have $v\in C$. So in order for $\mathcal{X}$ to cross-intersect $\cA_0$, all edges of $\mathcal{X} $ contain both $u$ and $v$.

For this fixed choice of $u,v\in M$ satisfying $u,v\in C$ for all edges $C\in \mathcal{X}$,  
we show that
$\cA_1(u,v)$ and $\mathcal{X}$  are not cross-intersecting.
Suppose not. Let $\cA'_1(u,v) =\{D\setminus \{1\}: D\in \cA_1(u,v)\}$, so $\cA'_1(u,v)\subset \binom{[n-1]}{r-1}$.
Since $|\cA'_1(u,v)| = |\cA_1(u,v)| > a(u,v)+\binom{n-3}{r-3}$, there are more than $\binom{n-3}{r-3}$ edges of $\cA'_1(u,v)$ that contain neither $u$ nor $v$. Applying Theorem \ref{t-intersecting} 
with $n-1$ and $r-1$ playing the roles of $n$ and $r$ respectively, and with $t = 2$, we see that among these edges there are two edges  
$E,E'$ such that $|E\cap E'|\leq 1$. First suppose that $E\cap E'=\emptyset$. For any $C\in \mathcal{X}$, $C$ must contain both $u$ and $v$
and intersect each of $E$ and $E'$. This yields $|\mathcal{X}|\leq (r-1)^2\binom{n-4}{r-4}$,
contradicting our assumption about $|\mathcal{X} |$. Hence we may assume that $E\cap E'=\{w\}$ for some $w\notin \{1,u,v\}$. As usual, all members of $\mathcal{X}$ must contain
$u$ and $v$ and intersect $E$ and $E'$. Among them there are at most $\binom{n-3}{r-3}$
that contain $w$ and at most $(r-2)^2 \binom{n-4}{r-4}$ that do not contain $w$. Hence $|\mathcal{X}|\leq \binom{n-3}{r-3}+(r-2)^2\binom{n-4}{r-4}$, contradicting our
assumption about $|\mathcal{X}|$. \qed

\medskip

Symmetric with $\cA_1(x,y)$ defined above, let $\cB_1(x,y) =\{D(z): \binom{n}{r} - f(z)\leq  \max\{\binom{n-2}{r-2}+3r^2\binom{n-3}{r-3}, b(x,y)\} + \binom{n-3}{r-3}\}$.  
By a similar argument which we omit, we have the following.

\medskip

{\bf Claim 3.}
Let $\mathcal{X} \subseteq \overline{\cS_1}$ satisfy $|\mathcal{X}|>\binom{n-3}{r-3}+r^2\binom{n-4}{r-4}$. Then there exist elements $u',v'\in [n]$ such that $\cB_1(u',v')$   
and $\mathcal{X}$ are not cross-intersecting. 

\medskip

Now, let $\cC$ be the subcollection of $\binom{[n]}{r}$ of minimum size such that $f(\cC)$ is an interval immediately following 
$f(\cA)$, and $|\cC \cap \overline{\cS_1}| =  1 + \binom{n-3}{r-3}+r^2\binom{n-4}{r-4}$. 
Similarly, let $\cD$ be the subcollection of $\binom{[n]}{r}$ of minimum size such that $f(\cD)$ is an interval immediately preceding 
$f(\cB)$, and $|\cD \cap \overline{\cS_1}| =  1 + \binom{n-3}{r-3}+r^2\binom{n-4}{r-4}$.
When $n$ is sufficiently large, $\cC$ and $\cD$  are well defined and are disjoint.
By definition, 
\begin{equation} \label{ABCD}
|\cA\cup \cB\cup \cC \cup \cD|\leq \binom{n-1}{r-1}+2\binom{n-3}{r-3}
+2r^2\binom{n-4}{r-4} + 2.
\end{equation}

By Claim 2 applied to $\cD \cap \overline{\cS_1}$ in place of $\mathcal{X}$, there exist elements $u,v\neq 1$ such that 
some member $D\in \cD$ is disjoint from an $r$-set $E$ satisfying $f(E)\le \max\{\binom{n-2}{r-2}+3r^2\binom{n-3}{r-3}, a(u,v)\}+\binom{n-3}{r-3}$.
Note also that $f(D)\geq \binom{n}{r}-|\cB|-|\cD|$. 

Letting 
$$\ell=\binom{n-2}{r-2}+3r^2\binom{n-3}{r-3}.$$
we then have
\begin{equation} \label{lower1}
dilation(f)\geq |f(D) - f(E)|\geq \binom{n}{r}-|\cB|-|\cD|-\max\{\ell, a(u,v)\}-\binom{n-3}{r-3}.
\end{equation}
By a similar argument, for some elements $u',v'\neq 1$, we have
\begin{equation} \label{lower2}
dilation(f)\geq \binom{n}{r}-|\cA|-|\cC|-\max\{\ell, b(u',v')\}-\binom{n-3}{r-3}.
\end{equation}
Let 
$$\lambda_1=\frac{1}{2}(|\cA|+|\cB|+|\cC|+|\cD|), \quad  \mbox{ and }  \quad 
\lambda_2=\frac{1}{2}(\max\{\ell, a(u,v)\}+ \max\{\ell, b(u',v')\}).$$
By averaging \eqref{lower1} and \eqref{lower2}, we get

\begin{equation} \label{lower-average}
dilation(f)\geq \binom{n}{r}-\lambda_1-\lambda_2-\binom{n-3}{r-3}.
\end{equation}

By \eqref{ABCD},
\begin{equation} \label{bound1}
\lambda_1\leq \frac{1}{2} \binom{n-1}{r-1}+\binom{n-3}{r-3}
+r^2\binom{n-4}{r-4} + 1.
\end{equation}

Letting $m(r,n) = \frac{1}{2}[\binom{n-2}{r-2}+\binom{n-3}{r-2}+\binom{n-4}{r-2}+\binom{n-5}{r-2}]$, we show that $\lambda_{2}\le m(r,n)$.  Observe first 
that $\lambda_{2} =  \max \{ \frac{1}{2}(\ell+a(u,v),  \frac{1}{2}(\ell+b(u',v'),  \frac{1}{2}(a(u,v)+b(u',v')), \ell \}$, and we can bound 
the expressions in the braces as follows.  
Certainly $a(u,v)$ is no more than the total number of $r$-subsets of
$[n]$ that contain $1$ and at least one of $u,v$. Thus
$a(u,v)\le \binom{n-1}{r-1}-\binom{n-3}{r-1} = \binom{n-2}{r-2}+\binom{n-3}{r-2}$, 
and similarly for $b(u',v')$.  Thus 
 $\frac{1}{2}(\ell+a(u,v))\le \frac{1}{2}[2\binom{n-2}{r-2} + \binom{n-3}{r-2} + 3r^{2}\binom{n-3}{r-3}] < m(r,n)$ for large $n$.  The same bound 
 holds for $\frac{1}{2}(\ell+b(u',v'))$ symmetrically.  Now
$a(u,v)+b(u',v')$ is no more than the total number of $r$-subsets of
$[n]$ that contain $1$ and at least one of $u,v,u',v'$. Thus
$a(u,v) + b(u',v')\le \binom{n-1}{r-1}-\binom{n-5}{r-1}=\binom{n-2}{r-2}+\binom{n-3}{r-2}
+\binom{n-4}{r-2}+\binom{n-5}{r-2} = 2m(r,n)$, so $m(r,n)\geq \frac{1}{2}[a(u,v)+b(u',v')]$.  Finally,  
 since $\ell=\binom{n-2}{r-2}+3r^2\binom{n-3}{r-3}$, for sufficiently large $n$ we have
$m(r,n) > \frac{1}{2}[\binom{n-2}{r-2}+\binom{n-3}{r-2}+\ell]\geq
\ell$.  So we've shown that $\lambda_{2}\le m(r,n)$.  

The required lower bound on $dilation(f)$ is now obtained as follows.  Applying Lemma \ref{binomial-estimate} to the four binomial terms in $m(r,n)$, we get 
\begin{equation} \label{bound2}
\lambda_2\leq m(r,n)\leq 2\frac{n^{r-2}}{(r-2)!}-\frac{r+4}{(r-3)!}n^{r-3} + 8r^4 n^{r-4}.
\end{equation}
Hence by \eqref{lower-average}, \eqref{bound1}, and \eqref{bound2}, we have
\begin{eqnarray*}
dilation(f) &\geq& \binom{n}{r}-\lambda_1-\lambda_2-\binom{n-3}{r-3}.\\
&\geq& \binom{n}{r}-\frac{1}{2}\binom{n-1}{r-1}-2\binom{n-3}{r-3}-2\frac{n^{r-2}}{(r-2)!}+\frac{r+4}{(r-3)!} n^{r-3}-9r^4 n^{r-4}\\
&\geq&\binom{n}{r}-\frac{1}{2}\binom{n-1}{r-1}-2\frac{n^{r-2}}{(r-2)!}+\frac{r+2}{(r-3)!} n^{r-3}-9r^4 n^{r-4}.\\
\end{eqnarray*}

\end{proof}

\section{The upper bound}  

In this section we give a construction which yields our upper bound for $B(K(n,r))$.  We begin with some notation.  For any 
sequence $\{ i_{j} \}$ of $t$ integers in increasing order, $1\le i_{1} < i_{2} < \cdots < i_{t}\le n$, let $S_{i_{1}i_{2}\cdots i_{t}}$ be the collection 
of all $r$-sets in $\binom{[n]}{r}$ whose smallest $t$ elements are $i_{1}, i_{2}, \cdots , i_{t}$, so that $|S_{i_{1}i_{2}\cdots i_{t}}| = \binom{n-i_{t}}{r-t}$.  Occasionally we insert commas 
between successive $i_{j}$ for clarity; e.g. $S_{1,8,10}$ is the collection of $r$-sets in $\binom{[n]}{r}$ whose smallest three elements are $1,8$, and $10$. 
For any 
$S\subset \binom{[n]}{r}$ , let $\mathcal{I}(S) = \{ w\in [n]: w\in \bigcap_{v\in S}v \}$, the intersection of all $r$-sets in $S$.    

Let $f:V(K(n,r))\rightarrow [1,\binom{n}{r}]$ be a one to one map.  As a convenience, in this section we identify any subset $A\subseteq V(K(n,r))$ with $f(A)\subseteq [1,\binom{n}{r}]$,  and any 
subset $B\subseteq [1,\binom{n}{r}]$ with $f^{-1}(B)$.  For subsets $F,G\subset  [1,\binom{n}{r}]$,
we say that $G$ is a {\it right blocker} of $F$ if
 \newline(a) $G = [u,\binom{n}{r}]$ for some $\frac{1}{2}\binom{n}{r} < u < \binom{n}{r}$ (so $G$ is a terminal interval).
 \newline(b) $F$ and $G$ are cross intersecting; that is, $v\cap w\ne \emptyset$ for any $v\in F$, $w\in G$, viewing $v,w$ as $r$-subsets under the identification above.  
 In particular, $vw\notin E(K(n,r))$ for all $v\in F$ and $w\in G$.

 For any subset $F\subset V(K(n,r))$ and map $f$ as above, let $\partial(F) =$ max$\{ |f(v) - f(w)|: v\in F$ or $w\in F, vw\in E(K(n,r)) \}$.

 \begin{lemma} \label{blockedintervals} Let $f:V(K(n,r))\rightarrow [1,\binom{n}{r}]$ be a one to one map.
 Suppose $G$ is a right blocker for $F = [x,y]\subset [1,\binom{n}{r}]$. Then
 $\partial(F)\le$ max$\{y-1,\binom{n}{r} - (|G| + x)\}$.

 \end{lemma} 
 
 \begin{proof}  Consider an edge $vw$, $v\in F$.  If $f(w) < f(v)$, then since $f(v)\le y$ and $f(w)\geq 1$ we get    
 $|f(v) - f(w)|\le y-1$ in this case.  If $f(w) > f(v)$, then since $w\notin G$ we have 
  $f(w) \le \binom{n}{r} - |G|$ while $f(v) \geq x.$  It follows that  $|f(v) - f(w)|\le \binom{n}{r} - (|G| + x)$, as required.     \end{proof}

 In our next theorem we obtain an upper bound for $B(K(n,r))$ by construction.

\begin{theorem} \label{betterupperbound} Let $r\geq 3$ be a fixed positive integer. Then for large $n$ we have  
 
$$B(K(n,r))\le \binom{n}{r} - \frac{1}{2}\binom{n-1}{r-1} - 2\frac{n^{r-2}}{(r-2)!} + (r + 2)\frac{n^{r-3}}{(r-3)!} + O(n^{r-4}). $$   

 \end{theorem}

 \begin{proof} Let $L(n,r) =  \binom{n}{r} - \frac{1}{2}\binom{n-1}{r-1} - 2\frac{n^{r-2}}{(r-2)!} + (r + 2)\frac{n^{r-3}}{(r-3)!} + O(n^{r-4})$, the right side 
 of the inequality in the theorem.  We will define a map $f:V(K(n,r))\rightarrow [1, \binom{n}{r}]$ and 
 a partition of  $[1, \binom{n}{r}]$ into intervals such that $\partial(F)\le L(n,r)$ for each 
 interval F in the partition.  Since $dilation(f)$ is the maximum of $\partial(F)$ over all these $F$, we obtain a map $f$ 
 satisfying $dilation(f)\le L(n,r)$, as required.

 First we begin with a partition  $S_{1}  = S_{1}' \cup S_{1}''$ of $S_{1}$ in which $|S_{1}' | = \lfloor \frac{1}{2}|S_{1}| \rfloor $ and 
 $|S_{1}'' | = \lceil \frac{1}{2}|S_{1}| \rceil $, defined as follows.    
We include $S_{12}\cup S_{15}\cup S_{1,8,9}\cup S_{1,8,10}$ in $S_{1}'$, and also include $S_{13}\cup S_{14}\cup S_{167}\cup S_{168}$ in $S_{1}''$, and then 
 fill out the rest of $S_{1}'$ and $S_{1}''$ arbitrarily with points from $S_{1}$, up to size $\lfloor \frac{1}{2}|S_{1}| \rfloor$ for $S_{1}'$, and up to size 
 $\lceil \frac{1}{2}|S_{1}| \rceil$ for $S_{1}''$.  Formally, we let
 $S_{1}' = S_{12}\cup S_{15}\cup S_{189}\cup S_{1,8,10}\cup X'$ and  $S_{1}'' = S_{13}\cup S_{14}\cup S_{167}\cup S_{168}\cup X''$, 
 where $X'$ is any subset of $S_{1} - \{ S_{12}\cup S_{15}\cup S_{189}\cup S_{1,8,10}\cup S_{13}\cup S_{14}\cup S_{167}\cup S_{168} \}$ of size 
 $\lfloor \frac{1}{2}|S_{1}| \rfloor - | S_{12}\cup S_{15}\cup S_{189}\cup S_{1,8,10} |$, and where $X'' = S_{1} -  S_{1}' - \{S_{13}\cup S_{14}\cup S_{167}\cup S_{168}\}  $.

Now define the map $f:V(K(n,r)\rightarrow [1, \binom{n}{r}]$ by Table \ref{thelayout}, with the following meaning.  There are $29$ cells  
in this table, counting $R = V(K(n,r))-(S_{1}\cup S_{2}\cup S_{3})$ as a single cell using wraparound.  We call these cells  {\it blocks} of $f$.  Each 
block labeled $S_{ij}$, $S_{t}$, $S_{ijk}$, or $R$ in this Table indicates that $f(S_{ij})$, $f(S_{t})$, $f(S_{ijk})$, or $R$ (respectively) is an interval of length 
$|S_{ij}|$, $|S_{t}|$, $|S_{ijk}|$, or $|R|$ (respectively) in $[1, \binom{n}{r}]$.      
The order in which points of $S_{ij}$ (or of $S_{t}$, $S_{ijk}$, or $R$) are mapped to this 
interval is arbitrary.  So we view the blocks of $f$ interchangably either as subsets of $V(K(n,r))$ or as intervals in $[1, \binom{n}{r}]$ under the 
identification explained at the beginning of this section.  
The relative order in which these blocks are mapped to $[1, \binom{n}{r}]$ is indicated by the left to right order 
of their appearance in the Table \ref{thelayout}, where the second row of the Table is understood to follow the first row in left to right order.

We define the blocks $..S_{ij}$, $S_{ij}..$, $S_{1}''..$, and $..S_{1}'$ in Table \ref{thelayout} not explained above.   
A block $..S_{ij}$ (resp. $S_{ij}..$) indicates that for some subset $S\subset S_{ij}$ the image $f(S)$ occupies some set of consecutive blocks 
immediately preceding (resp. following) the block $..S_{ij}$ (resp. $S_{ij}..$), and that $f(S_{ij} - S)=..S_{ij}$ (resp. $S_{ij}..$).   
Furthermore $f(S_{ij} - S)$ 
is the interval of length 
$|S_{ij} - S|$ in the position of block $..S_{ij}$ (resp. $S_{ij}..$).  So altogether $f(S)$ together with $..S_{ij}$ (resp. $S_{ij}..$) is a consecutive set of blocks
of $f$ which constitute the image $f(S_{ij})$.      
As an example, consider the block $..S_{15}$.  Referring to Table \ref{thelayout}, we see that the role of $S$ here is played by $S = S_{156}\cup S_{157}$, so $f(S)$ occupies the two blocks 
immediately preceding the block $..S_{15}$.  Thus $..S_{15} = f(S_{15} - S)$ is the interval of length $|S_{15} - S|$ in the position of block $..S_{15}$; specifically, 
$..S_{15} = f(S_{15} - S) = [|S_{12}\cup S_{156}\cup S_{157}| + 1, |S_{12}\cup S_{15}|]$.  Note that 
$f(S)$ together with $..S_{15}$ is a sequence of three consecutive blocks constituting the image $f(S_{15})$.  Referring to Table \ref{thelayout}, we see further that    
$f(S_{156}) = [ |S_{12}|+1, |S_{12}|+|S_{156}| ]$, and $f(S_{157}) = [ |S_{12}|+|S_{156}| +1, |S_{12}|+|S_{156}|+|S_{157}| ]$. As another example, for the the block $S_{1}''..$, the role of $S$ is played 
by $S = S_{168}\cup S_{167}\cup S_{14}..\cup S_{146}\cup S_{145}\cup S_{13}$, so $f(S)$ occupies the six consecutive blocks immediately following the block 
$S_{1}''..$.  We have $S_{1}''.. = f(S_{1}'' - S) = [\binom{n}{r} - |S_{1}''| + 1, \binom{n}{r} - |S|]$, so $S_{1}''..$ is the first of seven consecutive blocks which altogether constitute 
$f(S_{1}'') = [\binom{n}{r} - |S_{1}''| + 1, \binom{n}{r}]$.  Similarly 
$..S_{1}'$ is the last of seven consecutive blocks constituting $f(S_{1}') = [1, |S_{1}'|]$.

 \begin{table} 
 \begin{center}
\hskip-1.0cm\begin{tabular}{|c|c|c|c|c|c|c|c|c|c|c|c|c|c|c|}
 \hline
 $S_{12}$ & $S_{156}$ & $S_{157}$ & $..S_{15}$ & $S_{189}$ & $S_{1,8,10}$ & $..S_{1}'$& $S_{346}$ & $S_{347}$ & $..S_{34}$ & $S_{356}$ & $S_{357}$ & $..S_{35}$ & $..S_{3}$ & $R$ \\
 \hline
 $R$ & $S_{2}..$ & $S_{23}..$ & $S_{236}$ & $S_{235}$ & $S_{25}..$ & $S_{259}$ & $S_{258}$ & $S_{1}''..$ & $S_{168}$ & $S_{167}$ & $S_{14}..$ & $S_{146}$ & $S_{145}$ & $S_{13}$ \\  
 \hline
 \end{tabular}
 \caption{The mapping $f$} \label{thelayout}
 \end{center}
 \end{table}

As mentioned above it suffices to show that $\partial(F)\le L(n,r)$ for each block $F$ of $f$.  Start with the special block $F = R$.  Note first 
that $|S_{3}| = \binom{n-3}{r-1} < \binom{n-2}{r-1} = |S_{2}|$, and $| |S_{1}'| - |S_{1}''| |\le 1$.  Thus Table  \ref{thelayout} shows that 
the initial subinterval of $[1, \binom{n}{r}]$ consisting of the $14$ blocks immediately preceding $R$ (i.e the interval $S_{1}'\cup S_{3}$) is shorter that the final subinterval 
of $[1, \binom{n}{r}]$ consisting of the $14$ blocks immediately following $R$ (i.e the interval $S_{1}''\cup S_{2}$).  Thus 
 $\partial(R)\le \binom{n}{r} - |S_{1}'| - |S_{3}| = \binom{n}{r} -  \frac{1}{2}\binom{n-1}{r-1} - \binom{n-3}{r-1} < L(n,r)$, as required

 Consider now any block $F\ne R$. Each such $F$ is contained in either 
 $[1, \frac{1}{2}\binom{n}{r}]$ or  $[ \frac{1}{2}\binom{n}{r}, \binom{n}{r}]$.  Since $\frac{1}{2}\binom{n}{r} < L(n,r)$, any edge $ab\in E(K(n,r))$, $a<b$, for which 
 $|f(a) - f(b)|\geq L(n,r)$ must satisfy $a\in [1,\frac{1}{2}\binom{n}{r}]$ and $b\in[ \frac{1}{2}\binom{n}{r}, \binom{n}{r}]$.  We are thus reduced to showing 
 $\partial(F)\le L(n,r)$ for each block $F= [x,y]\subset [1,\frac{1}{2}\binom{n}{r}]$.  There are $14$ such blocks, in fact the leftmost $14$ blocks in the first row of Table \ref{thelayout}.  
 Since $y < \frac{1}{2}\binom{n}{r} < L(n,r)$, by Lemma \ref{blockedintervals} 
 we are reduced to showing that for each of these blocks $F$ there is a right blocker $G$ for $F$ 
 such that $|G| + x \geq \frac{1}{2}\binom{n-1}{r-1} + 2\frac{n^{r-2}}{(r-2)!} - (r + 2)\frac{n^{r-3}}{(r-3)!} + O(n^{r-4})$.  Let $B(n,r)$ 
be the right side of the last inequality.

For each block $M$ of $f$ let $\widehat{M}$ denote the terminal subinterval of $[1,\binom{n}{r}]$ beginning with $M$. 
 As an example, we can use 
 Table \ref{thelayout} to see that $\widehat{S_{259}} = \{S_{259}\cup S_{258}\cup S_{1}''..\cup S_{168}\cup S_{167}\cup S_{14}..\cup S_{146}\cup S_{145}\cup S_{13}\}.$  
With this notation, for each block $F \subset [1, \frac{1}{2}\binom{n}{r}]$, Table \ref{shortenedblockers} gives a right blocker $G$ of $F$ (directly below $F$ in the Table)
in the form $G = \widehat{M}$ for some block $M$.  For example, the right blocker of $S_{157}$ given by Table \ref{shortenedblockers} is $\widehat{S_{235}}$.   
To verify that these $G$ are indeed right blockers, it remains only to check that $\{ F\}$ and 
the corresponding $G = \widehat{M}$ are cross intersecting families. 
To do this, it suffices to show that for each block $B$ 
appearing in $\widehat{M}$ we have $\mathcal{I}(B)\cap \mathcal{I}(F)\ne \emptyset$.  For example, when $F = S_{189}$, the corresponding right blocker 
for $F$ in Table \ref{shortenedblockers}  is $G = \widehat{S_{259}}$, and the blocks $B$ appearing in $G$ are given earlier in this paragraph.  Examining each of these blocks $B$  
(in left to right order) for the condition $\mathcal{I}(B)\cap \mathcal{I}(F)\ne \emptyset$ 
we get   
$9\in \mathcal{I}(S_{259})\cap \mathcal{I}(189), 8\in \mathcal{I}(S_{258})\cap \mathcal{I}(189),\cdots , 1\in \mathcal{I}(S_{13})\cap \mathcal{I}(189)$, 
as required.  We leave to the reader the similar verification that for blocks $F\ne R$, $F\subset [1, \frac{1}{2}\binom{n}{r}]$, we have that $\{ F\}$ and the 
corresponding $G = \widehat{M}$ given by Table \ref{shortenedblockers} are cross intersecting families.

Let then $F$ be any block of $f$ satisfying $F = [x, y] \subset [1,\frac{1}{2}\binom{n}{r}]$.  We now verify the property $|G| + x \geq B(n,r)$, where $G$  
is the right blocker of $F$ given in Table \ref{shortenedblockers}.  

First consider such blocks satisfying $F\ne S_{12}$. The crucial feature of $f$ which ensures this property for such blocks $F$ is that
the right blocker $G$ of $F$ given in Table \ref{shortenedblockers} satisfies
 \begin{equation} \label{bockerproperty}
 |G| + x  = 1 + |S_{1}''(or S_{1}')\cup S_{ab}\cup S_{cd}\cup S_{rst}\cup S_{r's't'}|, b + d = 7. 
\end{equation} 
where the five sets on the right side have pairwise empty intersection.  Suppose for a moment that this property holds for $F$ and its right blocker $G$.  The using 
Lemma  \ref{binomial-estimate} we obtain  $|G|+x = 1 + \frac{1}{2}\binom{n-1}{r-1} + \binom{n-b}{r-2} + \binom{n-d}{r-2} + \binom{n-t}{r-3} + \binom{n-t'}{r-3} =   
1 + \frac{1}{2}\binom{n-1}{r-1} + 2\frac{n^{r-2}}{(r-2)!} -  (\frac{b+d+r-3-2}{(r-3)!})n^{r-3} + O(n^{r-4}) = B(n,r)$ 
since $b+d = 7$, as required.  

It remains to verify that (\ref{bockerproperty}) holds for these $F$ (and their corresponding $G$).  We do this for three cases, and leave the verification of the others to the reader.
Consider first $F = S_{157}$.  From Table \ref{thelayout} we have $x = 1 +  |S_{12}\cup S_{156}|$.  Table \ref{shortenedblockers} gives the right blocker for $F$ 
given by $G =  \widehat{S_{235}} = S_{235}\cup S_{25}\cup S_{1}''$.  Thus 
$|G| + x = 1 + |S_{1}''\cup S_{25}\cup S_{12}\cup S_{156}\cup S_{235}|$, as required by (\ref{bockerproperty}).  As a second example let    
$F = S_{1,8,10}$.  Table \ref{thelayout} gives $x = 1 +  |S_{12}| + |S_{15}| + |S_{189}|$, and 
Table \ref{shortenedblockers} gives the right blocker $G = \widehat{S_{258}} = S_{258}\cup S_{1}''$ of $S_{1,8,10}$ of $F$.  Thus we obtain  
$|G|+x = 1 + |S_{1}''\cup S_{12}\cup S_{15}\cup S_{189}\cup S_{258}|$, thereby verifying (\ref{bockerproperty}) for $F = S_{1,8,10}$.  Finally consider 
$F = S_{347}$.  Using the Tables Table \ref{thelayout} and \ref{shortenedblockers}  we have $x = 1 + |S_{1}'\cup S_{346}|$, while 
the right blocker for $F$ is $G = \widehat{S_{167}} = S_{167}\cup S_{14}\cup S_{13}$.  So finally $|G| + x = 1 + |S_{1}'\cup S_{14}\cup S_{13}\cup S_{167}\cup S_{346}|$, 
as required by (\ref{bockerproperty}).    
We leave to the reader the similar 
verification that (\ref{bockerproperty}) holds for the remaining blocks in the first row of Table \ref{thelayout} satisfying $F\ne S_{12}$, $F\ne R$.

Finally in the case $F = S_{12}$, Table \ref{shortenedblockers} gives the right blocker $G = \widehat{S_{2}} = S_{2}\cup S_{1}''$.  Hence
 $|G| + x > |G| = |S_{2}| + |S_{1}''| = \binom{n-2}{r-1} +   \frac{1}{2}\binom{n-1}{r-1} > B(n,r)$.      \end{proof}
  
\begin{table}
\begin{center}
\hskip-1.0cm\begin{tabular}{|c||c|c|c|c|c|c|c|c|c|c|c|c|c|c|}
 \hline
 $F$ & $S_{12}$ & $S_{156}$ & $S_{157}$ & $..S_{15}$ & $S_{189}$ & $S_{1,8,10}$ & $..S_{1}'$& $S_{346}$ & $S_{347}$ & $..S_{34}$ & $S_{356}$ & $S_{357}$ & $..S_{35}$ & $..S_{3}$ \\
 \hline
 $G$ & $\widehat{S_{2}..}$ & $\widehat{S_{236}}$ & $\widehat{S_{235}}$ & $\widehat{S_{25}..}$ & $\widehat{S_{259}}$ & $\widehat{S_{258}}$ & $\widehat{S_{1}''..}$ & $\widehat{S_{168}}$ & $\widehat{S_{167}}$ & 
 $\widehat{S_{14}..}$ & $\widehat{S_{146}}$ & $\widehat{S_{145}}$ & $\widehat{S_{145}}$ & $\widehat{S_{13}}$ \\ 
 \hline
 \end{tabular}
 \caption{block $F$ of $f$, right blocker $G$ for $F$} \label{shortenedblockers}
 \end{center}
 \end{table}


\end{document}